\documentclass[final,3p,times]{elsarticle}

\usepackage{amssymb}
\usepackage [latin1]{inputenc}
\usepackage{algorithm}
\usepackage{algorithmic}
\usepackage{multirow}
\usepackage{xcolor}
\usepackage{hyperref}
\usepackage{breakurl}
\usepackage{booktabs,threeparttable,caption2}
\usepackage{amsthm,amsmath}
\usepackage{float}
\usepackage{mathrsfs}
\usepackage{latexsym}
\usepackage{tabularx}
 \usepackage{diagbox}
 \usepackage{rotating}
\usepackage{cleveref}
 \biboptions{numbers,sort&compress}
\hypersetup{
    colorlinks=true,
    linkcolor=blue,
    filecolor=blue,
    urlcolor=blue,
    citecolor=cyan}
\newtheorem{thm}{Theorem}
\newtheorem{remark}{Remark}
\newtheorem{lem}{Lemma}

\newdefinition{example}{Experiment}
\newdefinition{rmk}{Remark}

\newcaptionstyle{left}{
\usecaptionmargin\captionfont
{\flushleft\bfseries\captionlabelfont\captionlabel\par}
\mbox{\onelinecaption{\captiontext}{\captiontext}}
}

\begin{document}

\begin{frontmatter}


\title{Efficient algorithm for the oscillatory matrix functions}
\author[author1]{Dongping Li\corref{cor1}}
\ead{lidp@ccsfu.edu.cn}
\author[author1]{Xue Wang}
\author[author1]{Xiuying Zhang\corref{cor1}}
\ead{xiuyingzhang@ccsfu.edu.cn}
\cortext[cor1]{Corresponding author.}
\address[author1]{Department of Mathematics, Changchun Normal University, Changchun 130032, PR China}

\begin{abstract}
This paper introduces an efficient algorithm for computing the general oscillatory matrix functions. These computations are crucial for solving second-order semi-linear initial value problems. The method is exploited using the scaling and restoring technique based on a quadruple angle formula in conjunction with a truncated Taylor series.
The choice of the scaling parameter and the degree of the Taylor polynomial relies on a forward error analysis.
Numerical experiments show that the new algorithm behaves in a stable fashion and performs well in both accuracy and efficiency.
\end{abstract}


\end{frontmatter}

\section{Introduction}
In this study, we propose a numerical method for approximating the oscillatory matrix functions defined by 
\begin{equation}\label{1.0a}
\phi_{l}(A)=\sum\limits^{\infty}_{k=0}\frac{(-1)^{k}A^k}{(2k+l)!},~~A\in \mathbb{R}^{N\times N},~~l\in \mathbb{N},
\end{equation}
which possess an infinite radius of convergence and are commonly denoted as $\phi$-functions.
These functions satisfy the recurrence relations
\begin{equation}\label{1.0b}
\phi_{l}(A)=\frac{1}{l!}-A\phi_{l+2}(A).
\end{equation}
Additionally, they can be reexpressed equivalently through the following integrals:
\begin{equation}\label{1.0d}
\phi_{l+1}(A)=\frac{1}{l!}\int_0^1\tau^l\phi_{0}\left((1-\tau)^2A\right)d\tau
\end{equation}
or
 \begin{equation}\label{1.0c}
\phi_{l+2}(A)=\frac{1}{l!}\int_0^1(1-\tau)\tau^l\phi_{1}\left((1-\tau)^{2}A\right)d\tau.
\end{equation}

Matrix functions of this class naturally emerge in the solution or numerical integration of second-order initial value
problem of the form
\begin{equation}\label{1.1}
\begin{cases}
y''(t)=-Ay(t)+f(t, y(t),y'(t)),~~ t\in[t_0,T],\\
y(t_0)=y_0,~~ y'(t_0)=y'_0.
\end{cases}
\end{equation}
Here, $y:\mathbb{R}\rightarrow \mathbb{R}^N$, $f:\mathbb{R}\times\mathbb{R}^N\times \mathbb{R}^N\rightarrow \mathbb{R}^N$. 
For instance, if $f\equiv 0$, the solution of problems (\ref{1.1}) can be expressed as:
\begin{equation}\label{1.2}
y(t)=\phi_{0}((t-t_0)^{2}A)y_0+(t-t_0)\phi_{1}((t-t_0)^{2}A)y'_0.
\end{equation}
More generally, under suitable assumptions on the smoothness of the right-hand side $f$, the exact solution of the system (\ref{1.1}) and its derivative are given by the variation of the constants formula \cite{grimm2006, Wu2015}
\begin{equation}\label{1.3}
\begin{cases}
y(t)=\phi_{0}((t-t_0)^{2}A)y_0+(t-t_0)\phi_{1}((t-t_0)^{2}A)y'_0+\int_{t_0}^{t}(t-\tau)\phi_{1}((t-\tau)^{2}A)f(\tau,y(\tau),y'(\tau))d\tau,\\
y'(t)=-(t-t_0)A\phi_{1}((t-t_0)^{2}A)y_0+\phi_{0}((t-t_0)^{2}A)y'_0+\int_{t_0}^{t}\phi_{0}((t-\tau)^{2}A)f(\tau,y(\tau),y'(\tau))dt.
\end{cases}
\end{equation}
In recent years, leveraging the variation of the constants formula (\ref{1.3}),  a broad family of structure-preserving numerical schemes have been exploited to compute the numerical solution of problem (\ref{1.1}). These methods involve more general oscillatory matrix functions (\ref{1.0a}) within their formulations. This is the main reason that we need to efficiently and accurately compute these functions. For an in-depth introduction and the latest advancements in structure-preserving methods for (\ref{1.1}), we direct readers to the monographs \cite{Wu2013,Wu2015}, and the references therein.

As shown above, accurate and efficient evaluation of oscillatory matrix functions is crucial for computing of the second-order initial value problem (\ref{1.1}). 
The first two matrix functions $\phi_0(A)$ and $\phi_1(A)$ can also be expressed in terms of the trigonometric matrix functions  $\texttt{sine}$ and $\texttt{cosine}$: 
\begin{equation}\label{1.4}
\begin{cases}
\phi_0(A)=\cos{(\sqrt{A})},\\
\phi_1(A)={(\sqrt{A})}^{-1}\sin{(\sqrt{A})}.
\end{cases}
\end{equation}
where $\sqrt{A}$ denotes any square root of $A$, see, e.g. \cite[Prob. 4.1]{Higham}. For singular $A$ this formula is interpreted by expanding ${(\sqrt{A})}^{-1}\sin{(\sqrt{A})}$ as a power series in $A$. 
Using (\ref{1.4}) and the double angle formulas of $\texttt{cosine}$ and $\texttt{sine}$, it is readily checked that the matrix functions $\phi_0(A)$ and $\phi_1(A)$ satisfy the relations:
\begin{equation}\label{1.5}
\begin{cases}
\phi_0(4A)=2\phi_0^2(A)-I,\\
\phi_1(4A)=\phi_0(A)\phi_1(A).
\end{cases}
\end{equation}

The computation of $\cos(A)$ and $\sin(A)$ has received significant research attention and several state-of-the-art algorithms have been provided in the numerical literature, see for instance \cite{Higham03,Higham05,Higham15,Sastre17,Sastre13,Alonso17,Defez19} and the references therein. Almost all the
widely used methods for computing the matrix trigonometric functions are the scaling and restoring technique combined with rational or polynomial approximations.
By employing the relation (\ref{1.4}), $\phi_0(A)$ and $\phi_1(A)$  can be computed through solving matrix
trigonometric functions. However, this approach necessitates explicitly calculating $\sqrt{A}$.

In \cite {AlMohy2017} Al-Mohy introduced an algorithm designed to compute the actions of $\phi_0(A)$ and $\phi_1(A)$ on vectors $B$. 
The algorithm first computes the scaled matrix functions $\phi_0(s^{-1}A)B$ and $\phi_1(s^{-1}A)B$ using truncated Taylor series, where $s$ is a nonnegative integer, and then applies a recursive procedure based on Chebyshev polynomials to restore the original matrix functions. 
In another work by Wu et al. \cite{Wu2020}  an algorithm based on quadrupling relations (\ref{1.5}) was developed to simultaneously compute  $\phi_0(A)$ and $\phi_1(A)$ . This algorithm approximates $\phi_0(4^{-s}A)$ and $\phi_1(4^{-s}A)$ using truncated Taylor approximations and then employs quadruple angle recurrence to recover the original matrix functions. 
Both algorithms rely on forward error analysis for parameters selection, with the former being suitable for large and sparse matrices, while the latter is more appropriate for medium and dense matrices.
To our knowledge, there have been few attempts to evaluate the more general cases thus far.
The aim of this paper is to propose a method for evaluating general $\phi$-functions. The method utilizes a scaling technique 
 based on a quadruple angle formula in conjunction with truncated Taylor approximations.
It determines both the scaling parameter and the Taylor degree through a forward error analysis. Numerical experiments demonstrate the reliability and effectiveness of the method.

The paper is structured as follows. In Section \ref{sec:2}, we introduce the algorithm for evaluating the general $\phi$-functions and provide a forward error analysis. The selection of the parameters involved is discussed in Section \ref{sec:3}. Section \ref{sec:4} presents numerical experiments to illustrate the performance of the algorithm. Finally, we draw some conclusions in Section \ref{sec:5}.

\section{Quadruple angle algorithm for $\phi_l(A)$}\label{sec:2}

When the norm of matrix $A$ is sufficiently small, the function $\phi_l(A)$ can be directly approximated using either Taylor or Pad\'{e} approximations. Nonetheless, such an approach become  impractical for matrices with large norms. This section presents the quadruple angle algorithm to compute $\phi_l(A)$. We begin our discussion by deriving a formula applicable to general $\phi$-functions.

\begin{lem} Given $A\in \mathbb{R}^{N\times N}$ and an integer~ $l\geq 2,$ then for any $a, b \in \mathbb{R},$ we have
\begin{equation}\label{2.2}
\begin{cases}
(a+b)^l\phi_l\left((a+b)^{2}A\right)=a^l\phi_{0}(b^{2}A)\phi_l\left(a^{2}A\right)+a^{l-1}b\phi_{1}(b^{2}A)\phi_{l-1}\left(a^{2}A\right)+\sum\limits^{l}_{k=2}\frac{1}{(l-k)!}a^{l-k}b^k \phi_k(b^2A),\\
(a+b)^{l-1}\phi_{l-1}\left((a+b)^{2}A\right)=-a^{l}bA\phi_{1}(b^{2}A)\phi_{l}\left(a^{2}A\right)+a^{l-1}\phi_{0}(b^2A)\phi_{l-1}\left(a^{2}A\right)+\sum\limits^{l}_{k=2}\frac{1}{(l-k)!}a^{l-k}b^{k-1}\phi_{k-1}(b^2A).
\end{cases}
\end{equation}
\end{lem}
\begin{proof} For any $u\in \mathbb{R}^{N},$ from (\ref{1.3}) we observe that
the solution and its derivative of second-order initial value problem
\begin{equation}\label{2.3}
\begin{cases}
y''(t)+Ay(t)=\frac{t^{l-2}}{(l-2)!}u,\\
y(0)=0,~~ y'(0)=0
\end{cases}
\end{equation}
at time $a+b$ are
\begin{equation}\label{2.4}
\begin{cases}
y(a+b)=\frac{1}{(l-2)!}\int_{0}^{a+b}(a+b-t)t^{l-2}\phi_{1}\left((a+b-t)^{2}A\right)dt\cdot v=(a+b)^{l}\phi_{l}\left((a+b)^{2}A\right)v,\\
y'(a+b)=\frac{1}{(l-2)!}\int_{0}^{a+b}t^{l-2}\phi_{0}\left((a+b-t)^{2}A\right)dt\cdot v=(a+b)^{l-1}\phi_{l-1}\left((a+b)^2A\right)v.
\end{cases}
\end{equation}

Alternatively, we can express the solution of Equation (\ref{2.3}) at time $a+b$ by employing a time-stepping method, achieved through dividing the time interval $[0, a+b]$ into two subintervals $[0, a]$ and $[a, a+b]$. At time $a,$ the solution and its derivative are 
\begin{equation}\label{2.4a}
\begin{cases}
y(a)=a^l\phi_l\left(a^{2}A\right)u,\\
y'(a)=a^{l-1}\phi_{l-1}\left(a^{2}A\right)u.
\end{cases}
\end{equation}
To advance the solution, utilizing $y(a)$, $y'(a)$ as initial value and again applying the formula (\ref{1.3}), we arrive at
\begin{equation}\label{2.5}
\begin{cases}
y(a+b)=\phi_{0}(b^{2}A)y(a)+b\phi_{1}(b^{2}A)y'(a)+\frac{1}{(l-2)!}\int_{a}^{a+b}t^{l-2}(a+b-t)\phi_{1}((a+b-t)^{2}A)udt,\\
y'(a+b)=-bA\phi_{1}(b^{2}A)y(a)+\phi_{0}(b^2A)y'(a)+\frac{1}{(l-2)!}\int_{a}^{a+b}t^{l-2}\phi_{0}((a+b-t)^{2}A)udt.
\end{cases}
\end{equation}
Substituting (\ref{2.4a}) into (\ref{2.5}) and  performing the integration by substitution for the integral in (\ref{2.5}), we obtain
\begin{equation}\label{2.6}
\begin{cases}
y(a+b)=a^l\phi_{0}(b^{2}A)\phi_l\left(a^{2}A\right)u+a^{l-1}b\phi_{1}(b^{2}A)\phi_{l-1}\left(a^{2}A\right)u+\sum\limits^{l}_{k=2}\frac{1}{(l-k)!}a^{l-k}b^k \phi_k(b^2A)u,\\
y'(a+b)=-a^lbA\phi_{1}(b^{2}A)\phi_{l}\left(a^{2}A\right)u+a^{l-1}\phi_{0}(b^2A)\phi_{l-1}\left(a^{2}A\right)u+\sum\limits^{l}_{k=2}\frac{1}{(l-k)!}a^{l-k}b^{k-1}\phi_{k-1}(b^2A)u.
\end{cases}
\end{equation}
By equalizing the expression (\ref{2.4}) with (\ref{2.6}), we directly establish the claim.
\end{proof}
In particular, the lemma yields the following quadruple angle formula
\begin{eqnarray}\label{2.7}
\phi_l(4A)=\frac{1}{2^l}\left(\phi_{0}(A)\phi_l(A)+\phi_{1}(A)\phi_{l-1}(A)+\sum\limits^{l}_{k=2}\frac{1}{(l-k)!}\phi_k(A)\right), ~~l\geq2. 
\end{eqnarray}
This formula forms the foundation of algorithm.

Let $s$ be a non-negative integer. We define  
\begin{eqnarray}\label{2.8}
X:=4^{-s}A~~\text{and}~~ C_{k,i}:=\phi_k(4^iX),~~k=0,1,\ldots,l,~~i=0,1,\ldots,s.
\end{eqnarray}
Utilizing the quadruple angle formula (\ref{2.7}) and starting with $C_{k,0},~k=0,1,\ldots,l$, we can compute $C_{l,s}=\phi_l(A)$ through the recurrence relation
\begin{equation}\label{2.9} 
C_{k,i}=\begin{cases} 2C_{0,i-1}^2-I, & k=0,\\
C_{0,i-1}C_{1,i-1}, & k=1,\\
\frac{1}{2^k}\left(C_{0,i-1}C_{k,i-1}+C_{1,i-1}C_{k-1,i-1}+\sum\limits_{j=2}^{k}\frac{1}{(k-j)!}C_{j,i-1}\right), & 1\leq k\leq 2
\end{cases} 
\end{equation}
for $i=1,2,\ldots,s$. 

We subsequently derive an absolute error bound for
this quadruple angle recurrence. Let $\widehat{C}_{k,0}$ denote an approximation of $C_{k,0}$, and $\widehat{C}_{k,i}$ for $1\leq i\leq s$ be generated from $\widehat{C}_{k,0}$ using the same recurrence as in (\ref{2.9}), namely,
\begin{equation}\label{2.10} 
\widehat{C}_{k,i}=\begin{cases} 
2{{\widehat{C}_{0,i-1}}}^2-I, & k=0,\\
 {\widehat{C}_{0,i-1}}{\widehat{C}_{1,i-1}}, & k=1,\\ \frac{1}{2^k}\left(\widehat{C}_{0,i-1}\widehat{C}_{k,i-1}+\widehat{C}_{1,i-1}\widehat{C}_{k-1,i-1}+\sum\limits_{j=2}^{k}\frac{1}{(k-j)!}\widehat{C}_{j,i-1}\right), & 2\leq k\leq l. \end{cases} 
 \end{equation}
Our objective is to establish bounds for the errors $E_{k,i}=C_{k,i}-\widehat{C}_{k,i}$ for all relevant $k$ and $i$.

\begin{thm}\label{th1} Let $E_{k,i}=\widehat{C}_{k,i}-C_{k,i}$, where $C_{k,0}=\phi_k(4^{-s}A),$ $\widehat{C}_{k,0}$ is an approximation of $C_{k,0}$, and $C_{k,i}$ and $\widehat{C}_{k,i}$ satisfy $(\ref{2.9})$ and  $(\ref{2.10})$, respectively. Assuming that $\|E_{k,i}\|\leq 0.05\|C_{k,i}\|$ for $ k=0,1$, then the errors can be bounded by

\begin{equation}\label{2.10c}
\|{E}_{k,i}\|\leq \begin{cases} 
4.1^{i} \prod \limits^{i-1}_{j=0} \max\limits_{0\leq \iota\leq k}\{\|C_{\iota,j}\|\}\cdot\max\limits_{0\leq \iota\leq k} \{\|E_{\iota,0}\|\}, & k=0,1,\\
\prod \limits^{i-1}_{j=0} \left(4.1 \max\limits_{0\leq \iota\leq k}\{\|C_{\iota,j}\|\}+0.25\right)\cdot\max\limits_{0\leq \iota\leq k} \{\|E_{\iota,0}\|\},&k\geq2.
 \end{cases} 
\end{equation}
\end{thm}

\begin{proof} Subtracting (\ref{2.9}) from (\ref{2.10}) gives the error recursion
\begin{equation}\label{2.11}
E_{k,i+1}=
\begin{cases}
2(C_{0,i}E_{0,i}+E_{0,i}C_{0,i}+E_{0,i}^2),&k=0,\\
C_{0,i}E_{1,i}+E_{0,i}C_{1,i}+E_{0,i}E_{1,i},&k=1,\\
\frac{1}{2^k}\left(C_{0,i}E_{k,i}+E_{0,i}C_{k,i}+C_{1,i}E_{k-1,i}+E_{1,i}C_{k-1,i}+E_{0,i}E_{k,i}+E_{1,i}E_{k-1,i}+\sum\limits^{k}_{j=2}\frac{1}{(k-j)!}E_{j,i}\right),&2\leq k\leq l.
\end{cases}
\end{equation}
Taking the norms of both sides of (\ref{2.11}), by the assumption on $\|E_{k,i}\|$ we obtain
\begin{equation}\label{2.12}
\|E_{k,i+1}\|\leq\begin{cases}
4.1\|C_{0,i}\|\cdot \|E_{0,i}\|,&k=0,\\
\|C_{1,i}\|\cdot \|E_{0,i}\|+1.05\|C_{0,i}\|\cdot \|E_{1,i}\|,&k=1,\\
\frac{1}{2^k}\left(\|C_{k,i}\|\cdot\|E_{0,i}\|+\|C_{k-1,i}\|\cdot\|E_{1,i}\|+1.05\|C_{1,i}\|\cdot\|E_{k-1,i}\|+1.05\|C_{0,i}\| \cdot\|E_{k,i}\|+\sum\limits^{k}_{j=2}\frac{1}{(k-j)!}\|E_{j,i}\|\right),& 2\leq k\leq l.
\end{cases}
\end{equation}
Define $c_{k,i}^\star=\max\limits_{0\leq j\leq k}\{\|C_{j,i}\|\}$ and $e_{k,i}=[\|E_{0,i}\|, \|E_{1,i}\|, \ldots,\|E_{k,i}\|]^T$. We have
\begin{equation}\label{2.13}
e_{k,i+1}\leq \Psi_{k,i} e_{k,i},
\end{equation}
where
\begin{equation}\label{2.14}
\Psi_{k,i}=\Theta_k(U_{k,i}+V_k)\in \mathbb{R}^{(k+1)\times (k+1)},
\end{equation}
with
\begin{equation}\label{2.15}
\Theta_k=\text{diag}(1,1,\frac{1}{2^2},\ldots,\frac{1}{2^k})\in \mathbb{R}^{(k+1)\times (k+1)},
\end{equation}
and
\begin{eqnarray}\label{2.16}
U_{k,i}=c_{k,i}^\star\left(
\begin{tabular}{cccccccc}
$4.1$ &$0$ &$0$ &$0$ &$0$  &$\cdots$ &$0$\\
$1$ &$1.05$ &$0$ &$0$ &$0$  &$\cdots$ &$0$ \\
$1$ &$2.05$ &$1.05$ &$0$ &$0$  &$\cdots$ &$0$\\
$1$ &$1$ &$1.05$ &$1.05$ &$0$  &$\cdots$ &$0$\\
$1$ &$1$ &$0$ &$1.05$ &$1.05$  &$\cdots$ &$0$\\
$\vdots$ &$\vdots$ &$\vdots$ &$\vdots$ &$\ddots$ &$\ddots$& $\vdots$\\
$1$ &$1$ &$0$ &0  &0 &$\cdots$ & $1.05$
\end{tabular}%
\right), V_k=\left(
\begin{tabular}{cccccccc}
$0$ &$0$ &$0$ &$0$ &$0$  &$\cdots$ &$0$\\
$0$ &$0$ &$0$ &$0$ &$0$  &$\cdots$ &$0$ \\
$0$ &$0$ &$1$ &$0$ &$0$  &$\cdots$ &$0$\\
$0$ &$0$ &$1$ &$1$ &$0$  &$\cdots$ &$0$\\
$0$ &$0$ &$\frac{1}{2!}$ &$1$ &$1$  &$\cdots$ &$0$\\
$\vdots$ &$\vdots$ &$\vdots$ &$\vdots$ &$\ddots$ &$\ddots$& $\vdots$\\
$0$ &$0$ &$\frac{1}{(k-2)!}$ & $\frac{1}{(k-3)!}$  &$\frac{1}{(k-1)!}$  &$\cdots$ & $1$
\end{tabular}%
\right).
\end{eqnarray}
From the recursion (\ref{2.13}) we have
\begin{eqnarray}\label{2.17}
\|e_{k,i+1}\|_{\infty}\leq \prod\limits^{i}_{j=0}\|\Psi_{k,j}\|_{\infty} \|e_{k,0}\|_{\infty}.
\end{eqnarray}
Notice that $\|\Psi_{k,j}\|_{\infty}=4.1c_{k,j}^\star$ for $k=0,1$, and $\|\Psi_{k,j}\|_{\infty}\leq 4.1c_{k,j}^\star+0.25$ for $k\geq 2$, this yields the required conclusion.
\end{proof}
\begin{remark}
It follows directly from (\ref{1.0d}) and (\ref{1.4}) that  $|\phi_k(x)|\leq \frac{1}{k!}$ for any non-negative real number $x$.
Thus, in the special case where the matrix $A$ is positive semi-definite matrix, we have that $\|C_{k,j}\|_2\leq 1$, and the bound (\ref{2.10c}) reduces to
\begin{equation}\label{2.17b}
\|{E}_{k,s}\|_2 \leq\begin{cases}
4.1^s \max\limits_{0\leq \iota\leq k} \{\|E_{\iota,0}\|_2\}, &k= 0, 1,\\
4.35^s\max\limits_{0\leq \iota\leq k} \{\|E_{\iota,0}\|_2\}, & k \geq 2.
\end{cases}
\end{equation}
Although the error bound may become considerable for large values of $s$, it remains rigorous. Consequently, if the error bound is sufficiently small, it ensures that the actual error is also small, often improving by multiple orders of magnitude. 
\end{remark}
\begin{remark}
Assuming that $\phi_{0}(4^{-s}A)$ and $\phi_1(4^{-s}A)$ can be computed exactly, though this assumption is practically infeasible from a numerical point of view, a proof similar to that in Theorem \ref{th1} reveals that the error bound to be
\begin{equation}\label{2.17d}
\|{E}_{k,i}\|\leq 
 \begin{cases} 
0, & k=0,1,\\
(\frac{1}{4})^i\prod \limits^{i-1}_{j=0} \left( \max\limits_{0\leq \iota\leq 1}\{\|C_{\iota,j}\|\}+1\right)\cdot\max\limits_{2\leq \iota\leq k} \{\|E_{\iota,0}\|\}, &k\geq2.
\end{cases}
\end{equation}
\end{remark}
To develop an algorithm, it is required to pre-evaluate $\phi_j(4^{-s}A)$ for $j=0,1,\cdots,l$. When the norm of $4^{-s}A$ is sufficiently small, rational and polynomial approximations can be used to compute such matrix functions. Recent studies have indicated that Taylor-based approximations may exhibit higher efficiency compared to approximations \cite{Sastre17,Ruiz16}. Therefore, within this framework, we opt for Taylor-based approximations to compute the $\phi_j(4^{-s}A)$. 

Denote
\begin{eqnarray}\label{2.18a}
T_{j,m}(x):=\sum\limits^{m}_{k=0}\frac{(-1)^k}{(2k+j)!}x^k,~~ j=0,1,\cdots,l
\end{eqnarray}
as the truncated Taylor series of order $m$ for the function $\phi_j(x)$. 
The nonnegative integer $s$ is chosen such that $\phi_j(4^{-s}A)$ is well-approximated by $\widehat{C}_j:=T_{j,m}(4^{-s}A)$.
By applying the quadruple angle formula $s$ times iteratively, one can derive approximations to $\phi_j(A)$ for $j=0,1,\cdots,l$. Algorithm \ref{alg1.1} provides a concise outline of the procedure for computing general oscillatory matrix functions.

The matrix polynomials $T_{j,m}(4^{-s}A)$ for $j=0,1,\ldots,l$ can be computed using the Paterson-Stockmeyer (\texttt{PS}) method \cite[p. 72-74]{Higham}, \cite{Paterson}, which is a widely used general technique for evaluating matrix polynomials. Our tests indicate that the \texttt{PS} method attains higher  accuracy than the explicit powers method \cite[Algorithm 4.3]{Higham}, although the latter may involve fewer matrix-matrix products when simultaneously computing all the $l + 1$ matrix polynomials. To fully exploit the performance of the \texttt{PS} method, as illustrated in \cite[p. 74]{Higham}, we constrain the polynomial degree $m$ to the optimal set 
\begin{eqnarray*}
\mathbb{M}=\{2, 4, 6, 9, 12, 16, 20, 25, 30, 36,\ldots\}.
\end{eqnarray*}
Algorithm \ref{alg1.0} presents the pseudocode for applying the \texttt{PS} methodology to compute  $T_{j,m}(4^{-s}A)$ for $j=0,1,\ldots,l$. The process entails $\pi_m=\left(\lceil\sqrt {m}~\rceil-1\right)+(l+1)\left ( m/\lceil\sqrt {m}~\rceil-1\right)$ matrix-matrix products. Initially, the algorithm computes powers $A^i$ for $2\leq i\leq q$, which can be done during the parameter selection phase (Step 1 of Algorithm \ref{alg1.1}). Further details will be elucidated in the ensuing section. Excluding the calculation of $A^i$ for $2\leq i\leq q$, the procedure is amenable to parallel implementation. 

\begin{algorithm}[htb]
\caption{~$\texttt{quadphi}$: the quadruple angle algorithm for computing $\phi_l(A)$, $j=0,1,2, \ldots,l$.}\label{alg1.1}
\begin{algorithmic}[1]
\REQUIRE $A \in \mathbb{C}^{N\times N},$ $l$
\STATE Select optimal values of $m$ and $s$
\STATE $X=4^{-s}A$
\STATE Compute $\widehat{C}_j=T_{j,m}(X)$, $j=0,1,2, \ldots,l$
\IF {$s=0$} \RETURN $\widehat{C}_j$, $j=0,1,2, \ldots,l$ \ENDIF
\FOR{$i=1:s$}
\IF {$l=0$} \RETURN $\widehat{C}_j$, $j=0,1,2, \ldots,l$ \ENDIF
\STATE Compute $\widehat{C}_0=2{\widehat{C}_0}^2-I$
\STATE Compute $\widehat{C}_1=\widehat{C}_0\widehat{C}_1$
\STATE Compute $\widehat{C}_k=\frac{1}{2^k} \left(\widehat{C}_0\widehat{C}_k+\widehat{C}_1\widehat{C}_{k-1}+\sum\limits^{k}_{j=2}\frac{1}{(k-j)!}\widehat{C}_j\right),~k=2,\ldots,l$
\ENDFOR
\ENSURE~$\widehat{C}_k$,  $k=0,1,\ldots,l$
\end{algorithmic}
\end{algorithm}

\begin{algorithm}[htb]
\caption{ The \texttt{PS} method for computing matrix polynomials $T_{j,m}(A)$ for $j=0,1,2, \ldots,l$.}\label{alg1.0}
\begin{algorithmic}[1]
\REQUIRE $ A\in \mathbb{C}^{N\times N},$ $l$, $m\in\mathbb{M}=\{2, 4, 6, 9, 12, 16, 20, 25,\ldots\}$
\STATE $q= \lceil\sqrt{m}~\rceil$,  $r= \lfloor m/q \rfloor$
\STATE Compute $A_i=A^i,~~i=1,2,\ldots,q$
\STATE Compute $T_{j,m}=\sum\limits^q_{i=0}\frac{1}{\left(2(m-q+i)+j\right)!}A_i$, $j=0,1,\ldots,l$
\FOR{$k=r-2:0$}
\STATE Compute $T_{j,m}=T_{j,m}A_q+\sum\limits^{q-1}_{i=0}\frac{1}{\left(2(qk+i)+j\right)!} A_i$, $j=0,1,\ldots,l$
\ENDFOR
\ENSURE~$T_{j,m}$, $j=0,1,\ldots,l$
\end{algorithmic}
\end{algorithm}

\section{Determination of the scaling parameter $s$ and the Taylor degree $m$ }\label{sec:3}
Next, we address the selection of the scaling parameter $s$ and the Taylor degree  $m$. For a given tolerance $\texttt{Tol}$, 
the scaling parameter $s$ and the Taylor degree $m$ should be chosen to satisfy the condition: 
\begin{equation}
\|\phi_j(X)-T_{j,m}(X)\|=\|\sum\limits^{\infty}_{k=m+1}\frac{1}{(2k+j)!}X^k\| \leq \texttt{Tol}, ~~j=0,1,\ldots,l.
\end{equation}
Furthermore, let we define the function 
\begin{equation}\label{3.1}
 h_{m}(\theta):=\sum\limits^{\infty}_{k=m+1}\frac{1}{(2k)!}\theta^k.
\end{equation}
According to the theorem  presented in \cite[Thm. 4.2(a)]{AlMohy2009},  we have
\begin{equation}\label{3.2}
\|\phi_j(X)- T_{j,m}(X)\|\leq\sum\limits^{\infty}_{k=m+1}\frac{1}{(2k)!}\|X^k\|\leq h_m\left(\alpha_p(X)\right),
\end{equation}
where $\alpha_p(X)=\max\{\|X^p\|^{1/p}$, $\|X^{p+1}\|^{1/(p+1)}\}$, and $p(p-1)\leq m+1$.

Define $\eta_m(X)= \min\{\alpha_p(X)~|~ p(p-1)\leq m + 1\}$ and let $\theta_{m}$ denote the largest value of $\theta$ such that the bound in (\ref{3.2}) does not exceed the tolerance  $\texttt{Tol}$, i.e.,
\begin{equation}\label{3.3}
\theta_{m}=\max{\{\theta~|~{h_{m}(\theta)}\leq \texttt{Tol}\}}.
\end{equation}
Thus, once the scaling parameter $s$ is selected to satisfy
\begin{equation}\label{3.4}
\eta_m(X)\leq \theta_{m},
\end{equation}
we have
\begin{eqnarray}\label{3.5}
\|\phi_j(X)- T_{j,m}(X)\|\leq \texttt{Tol}.
\end{eqnarray}
 From (\ref{3.4}) we have
\begin{equation}\label{3.6}
s\geq\log_4\left(\eta_m(A)/\theta_{m}\right).
\end{equation}
Naturally, we choose the smallest non-negative integer $s$ such that the inequality (\ref{3.6}) holds. 

In practice, the value of $\theta_{m}$ can be evaluated by substituting the $h_{m}(\theta)$ with its first $\nu$-terms truncated series and then solving numerically the algebra equation
\begin{equation}\label{3.7}
\sum\limits^{\nu+m}_{k=m+1}\frac{1}{(2k)!}\theta^k=\texttt{Tol}.
\end{equation}
Table \ref{tab2.1} lists the evaluations of $\theta_{m}$ for $m=1:20$ when $\nu=150$ and $\texttt{Tol}=2^{-53}\approx1.1\cdot10^{-16}$. 

\begin{table}[h]
\setlength{\abovecaptionskip}{0.cm}
\setlength{\belowcaptionskip}{-0.3cm}
\caption{The first 20 values of $\theta_{m}$ satisfy (\ref{3.7}) when $\texttt{Tol}=2^{-53}$.}\label{tab2.1}
\begin{center}
\begin{tabular*}{\textwidth}{@{\extracolsep{\fill}}@{~}c|ccccccccccccc}
\toprule%
{$m$} & $1$& $2$& $3$& $4$& $5$& $6$& $7$ & $8$ & $9$ & $10$\\
  \hline
  $\theta_{m}$ & $5.16\text{e-}8$ & $4.31\text{e-}5$ & $1.45\text{e-}3$ &$1.32\text{e-}2$ &$6.13\text{e-}2$&$1.92\text{e-}1$&$4.68\text{e-}1$
  &$9.63\text{e-}1$ &$1.75$ &$2.90$\\
 \toprule%
  {$m$}& $11$& $12$& $13$ & $14$& $15$& $16$& $17$& $18$& $19$& $20$ \\
  \hline
  $\theta_{m}$&$4.50$ &$6.59$&$9.25$ & $12.52$ & $16.45$ & $21.09$ &$26.46$&$32.61$
  &$39.57$ &$47.35$ \\
\bottomrule
\end{tabular*}
\end{center}
\end{table}
Here, we present a specific approach for determining the scaling parameter $s$ and the Taylor degree $m$.
First, we sequentially select the values of $m$ from the set $M=\{2, 4, 6, 9, 12, 16, 20\}$. If there exists an $m$ such that $\eta_m(A)\leq \theta_m$, we set $s = 0$;
 Otherwise, if $\eta_{20} > \theta_{20}$,  we set $m = 20$ and $s=\lceil\log_4\left(\eta_{20}/\theta_{20}\right)\rceil$. 
The pseudocode of this process is summarized in Algorithm \ref{alg2}. Within this algorithm, 
to avoid computing additional matrix powers, we bound some $\|A^i\|$ from the products of norms of matrices that have been
previously computed. For instance,  we use $\min(\|A\|\cdot\|A^4\|,\|A^2\|\cdot\|A^3\|)$ to bound $\|A^5\|$. 
Integrating Algorithm \ref{alg1.0} and Algorithm \ref{alg2} into Algorithm \ref{alg1.1} enables the evaluation of $\phi_{\ell}(A)$.

\begin{algorithm}[htb]
\caption{~This algorithm returns the parameters $m,~s$, and computes the powers $A_i=A^i/4^{is}$, $1\leq i\leq q$.}
\label{alg2}
\begin{algorithmic}[1]
\REQUIRE~$A\in \mathbb{R}^{N\times N},$ $l$.
\STATE $s=0$ 
\STATE $A_1=A,$ $d_1=\|A_1\|_1$
\STATE  if {$d_1\leq\theta_1$,} $m=1$, quit, end if
\STATE $A_2=A^2,$ $d_2=\|A_2\|_1$.
\STATE $\alpha_2=\max(d_2^{1/2},(d_1*d_2)^{1/3}),$ $\eta=\alpha_2$
\STATE if\ {$\eta<=\theta_2,$} $m=2,$ quit, end if
\STATE if\ {$\eta<=\theta_4,$} $m=4,$ quit, end if
\STATE $A_3=A_1 A_2,$ $d_3=\|A_3\|_1,$ $d_4=\min(d_1d_3,{d_2}^2),$ $\alpha_2=\max(d_2^{1/2},d_3^{1/3}),$ $\alpha_3=\max(d_3^{1/3},d_4^{1/4})$
\STATE $\eta=\min(\alpha_2,\alpha_3).$
\STATE if\ {$\eta<=\theta_6,$} $m=6,$ quit, end if
\STATE if\ {$\eta<=\theta_9,$} $m=9,$ quit, end if
\STATE $A_4=A_2^2,$ $d_4=\|A_4\|_1,$ $d_5=\min(d_1d_4,d_2d_3),$ $\alpha_3=\max(d_3^{1/3},d_4^{1/4}),$ $\alpha_4=\max(d_4^{1/4},d_5^{1/5})$
\STATE $\eta=\min(\alpha_2,\alpha_3,\alpha_4)$
\STATE if\ {$\eta<=\theta_{12},$} $m=12,$ quit, end if
\STATE if\ {$\eta<=\theta_{16},$} $m=16,$ quit, end if
\STATE $A_5=A_1 A_4$, $d_5=\|A_5\|_1$,  $d_6=\min([d_1*d_5,d_2*d_4,d_3^2])$,  $\alpha_4=\max([d_4^{1/4},d_5^{1/5}])$, $\alpha_5=\max([d_5^{1/5},d_6^{1/6}])$
\STATE $\eta=\min([\alpha_2, \alpha_3, \alpha_4, \alpha_5])$
\STATE if\ {$\eta<=\theta_{20}$},  $m=20,$ quit, end if
\STATE $m=20,$ $s= \lceil\frac{1}{2}\log_2(\eta/\theta_{20})\rceil$
\STATE $q= \lceil\sqrt{m}~\rceil$
\FOR{$i=1:q$}
\STATE $A_i=A_i/4^{is}$
\ENDFOR
\ENSURE~$m,s, A_i$
\end{algorithmic}
\end{algorithm}

\section{Numerical experiments}\label{sec:4}
This section presents two numerical experiments to illustrate the performance of the new algorithm \texttt{quadphi}. The MATLAB codes of the algorithm
are available at \url{https://github.com/lidping/quadphi.git}. This routine can generate the values of oscillatory matrix functions for multiple indices simultaneously, with the computational workload increasing by only one matrix-matrix multiplication for each additional output.
 All tests are conducted in MATLAB R2020b running on a desktop equipped with an Intel Core i7 processor running at 2.1 GHz and 64 GB of RAM.

\begin{example} \label{exa1}
The experiment focuses on testing the stability of \texttt{quadphi}. It involves a set of 83 test matrices, comprising  51 $10 \times 10$
matrices obtained from the function \texttt{matrix} in the Matrix Computation Toolbox \cite{Highamtool}, together with the 32 test matrices  of sizes ranging from $2 \times 2$ to $20 \times 20$ used in \cite{AlMohy2022}. The ranges of the 2-norms and the 1-norms for these matrices span from  $4.1\cdot10^{-6}$ to $10^{17}$. We evaluate the relative error $\texttt{Error}=\|\phi_l(A)-\widehat{C}_l\|_1/\|\phi_l(A)\|_1$, where $\widehat{C}_l$ represents the computed solution. The "exact" values of the $\phi$-functions for these matrices are computed using 150-digit arithmetic via the Symbolic Math Toolbox in MATLAB.

In Fig. \ref{fig4.1} we plot eight precision diagrams illustrating the performance of \texttt{quadphi} in computing $\phi_l(A)$ for $l= 0,1,\ldots,7$, arranged from left to right. For each matrix $A$, all the eight matrix functions are computed simultaneously by a single invocation of \texttt{quadphi}.
The solid black line in each diagram represents the product of the unit roundoff \texttt{eps} and the relative condition number \texttt{cond}. The \texttt{cond} is estimated using the code \texttt{funm\underline{~}condest1} from the Matrix Function Toolbox \cite{Highamtool} with 150-digit precision using MATLAB's Symbolic Math Toolbox. For a stable algorithm the relative errors should closely follow the solid black line. It is observed that the relative error is generally remains below the solid line, indicating that \texttt{quadphi} behaves in a numerically stable manner for all matrices.
\end{example} 

\begin{example} \label{exa2} In this experiment, we aim to assess the accuracy and efficiency of our algorithm \texttt{quadphi} by solving $\phi_l(A)b$, where $b$ is a vector with all elements equal to one. We use a set of 141 matrices, each with a dimension of $n=128$. Among these, 41 matrices are generated using the MATLAB routine \texttt{matrix} from the Matrix Computation Toolbox \cite{Highamtool}, while the remaining 100 are randomly generated. Fifty percent of the randomly generated matrices are diagonalizable, with the other half being non-diagonalizable.

We calculate $\varphi_l(A)b$ by first computing $\varphi_l(A)$ using \texttt{quadphi}, then  forming the product of $\varphi_l(A)$ and $b$. We have also included a comparison with the MATLAB built-in function \texttt{ode45}. We set an absolute tolerance of $10^{-20}$ and a relative tolerance of $2.22045\cdot10^{-14}$ for \texttt{ode45} to solve the corresponding second-order initial value problem of $\phi_l(A)b$. We evaluate the relative error in the 2-norm of the computed vectors. The "exact" $\phi_l(A)b$ is computed in 150 significant decimal digit arithmetic using MATLAB's Symbolic Math Toolbox.

In Fig \ref{fig4.2}, from left to right we present the relative errors for \texttt{quadphi} and \texttt{ode45} when solving $\phi_l(A)b$, $l= 0,1,\ldots,7$, respectively, for each of the 141 test matrices. It is observed that \texttt{quadphi} generally achieves better accuracy. Specifically, compared with \texttt{ode45}, \texttt{quadphi} achieves higher accuracy for 138, 136, 134, 131, 131, 94, 132 out of the 141 matrices. Table \ref{tab2.1} displays the overall execution times of \texttt{quadphi} and \texttt{ode45} for the 141 test matrices, for each of $\phi_l(A)b$. As shown, \texttt{quadphi} notably requires less CPU time than \texttt{ode45}. 
\end{example} 

\section{Conclusion}\label{sec:5}
This paper presents an efficient method for computing general oscillatory matrix functions $\varphi_l(A)$. We have 
developed the quadruple formulas for these functions, upon which the method is constructed using the scaling and restoring technique in conjunction with truncated Taylor series. The scaling parameter and the Taylor degree are determined through forward error analysis.
 The algorithm is applicable for any matrix $A$, with its computational cost primarily driven by matrix-matrix multiplications.
A MATLAB implementation, \texttt{quadphi} based on that method has been developed and tested. Numerical experiments demonstrate that \texttt{quadphi} is stable and efficient. Future work will focus on evaluating the actions of the $\varphi$-function of large and sparse matrix on vectors.     

\section*{Acknowledgements}
 This work was supported by the Natural Science Foundation of Jilin Province (Grant No. JJKH20240999KJ) and the Natural Science Foundation of Changchun Normal University (Grant No. 2021001).

\begin{figure}[H]
\begin{minipage}{0.5\linewidth}
\centering
\includegraphics[width=6.5cm,height=5cm]{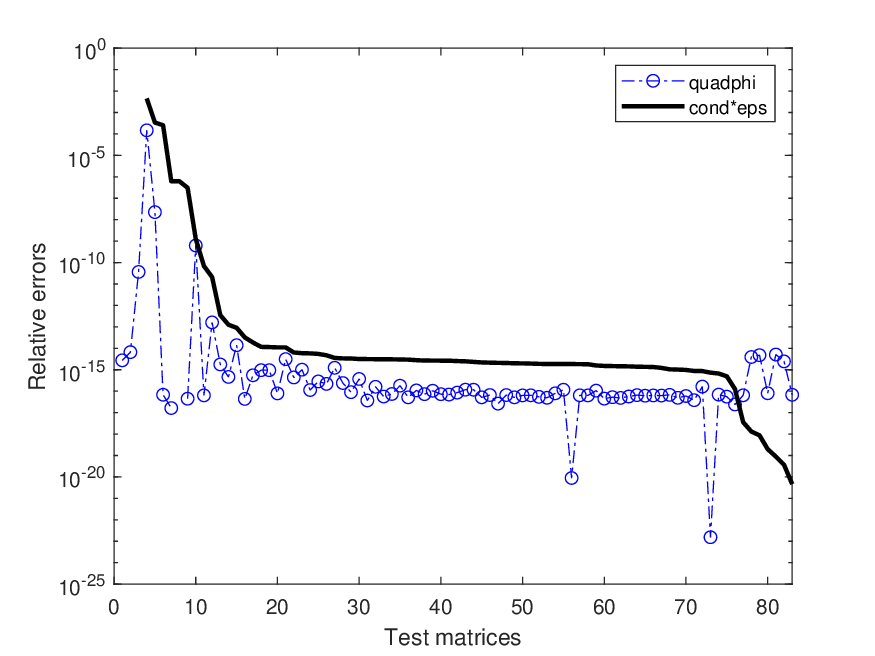}\\
\end{minipage}
\mbox{\hspace{-1.5cm}}
\begin{minipage}{0.5\linewidth}
\centering
\includegraphics[width=6.5cm,height=5cm]{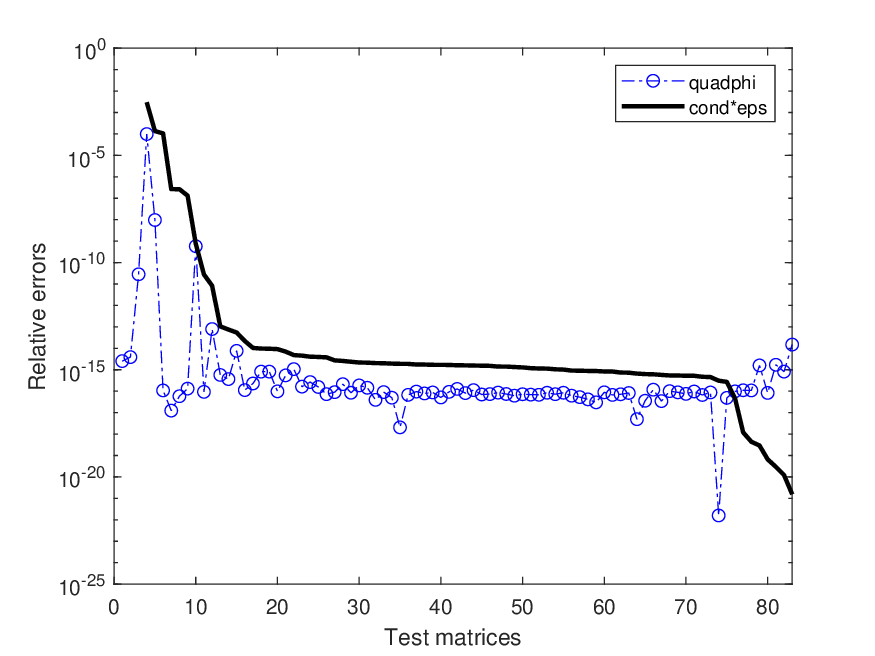}\\
\end{minipage}\\
\begin{minipage}{0.5\linewidth}
\centering
\includegraphics[width=6.5cm,height=5cm]{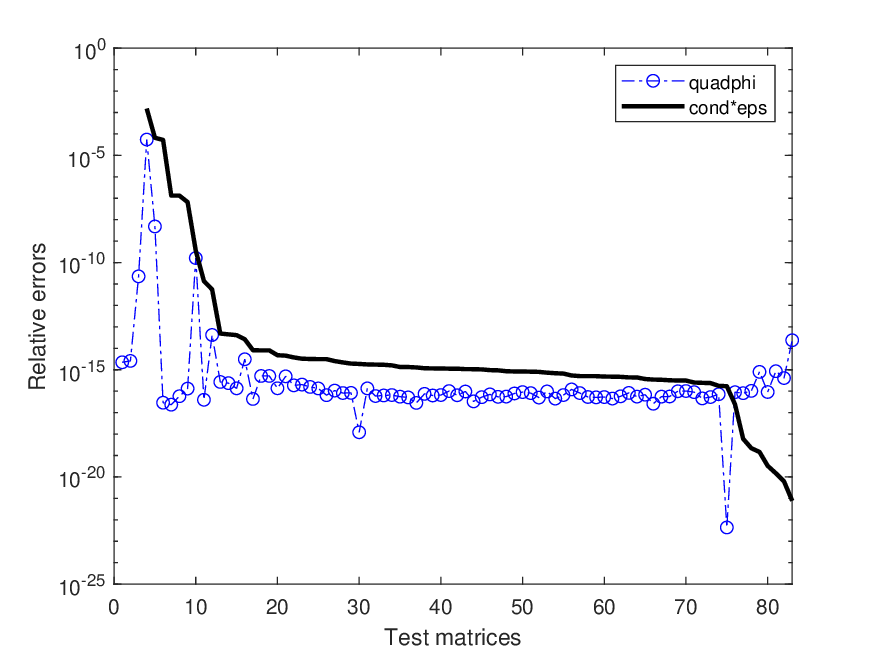}\\
\end{minipage}
\mbox{\hspace{-1.5cm}}
\begin{minipage}{0.5\linewidth}
\centering
\includegraphics[width=6.5cm,height=5cm]{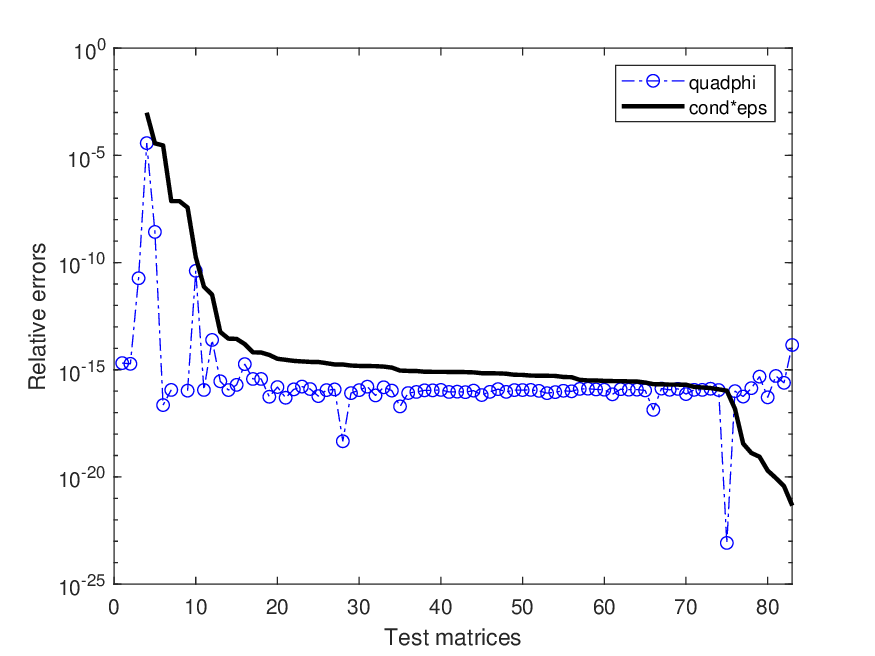}\\
\end{minipage}\\
\begin{minipage}{0.5\linewidth}
\centering
\includegraphics[width=6.5cm,height=5cm]{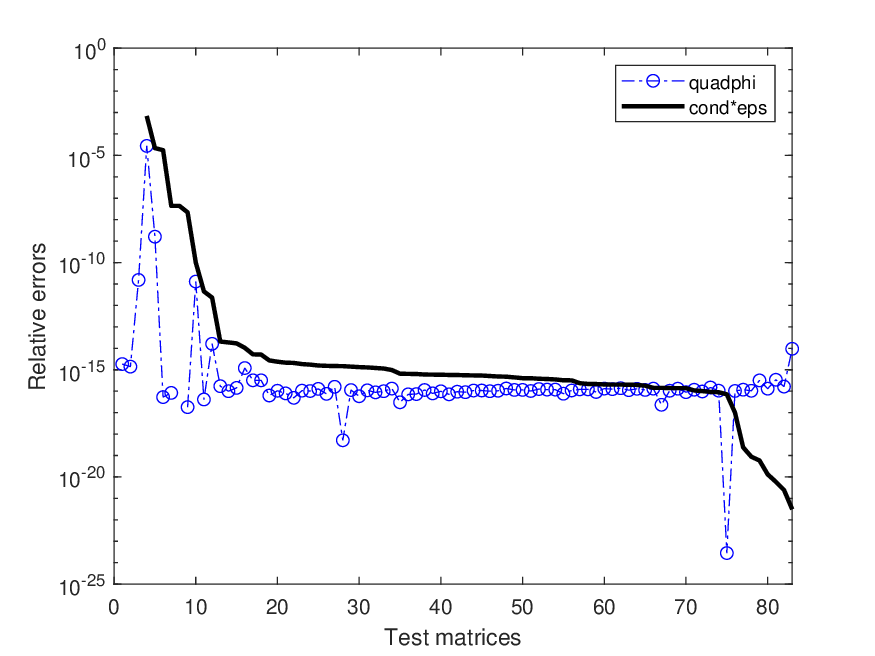}\\
\end{minipage}
\mbox{\hspace{-1.5cm}}
\begin{minipage}{0.5\linewidth}
\centering
\includegraphics[width=6.5cm,height=5cm]{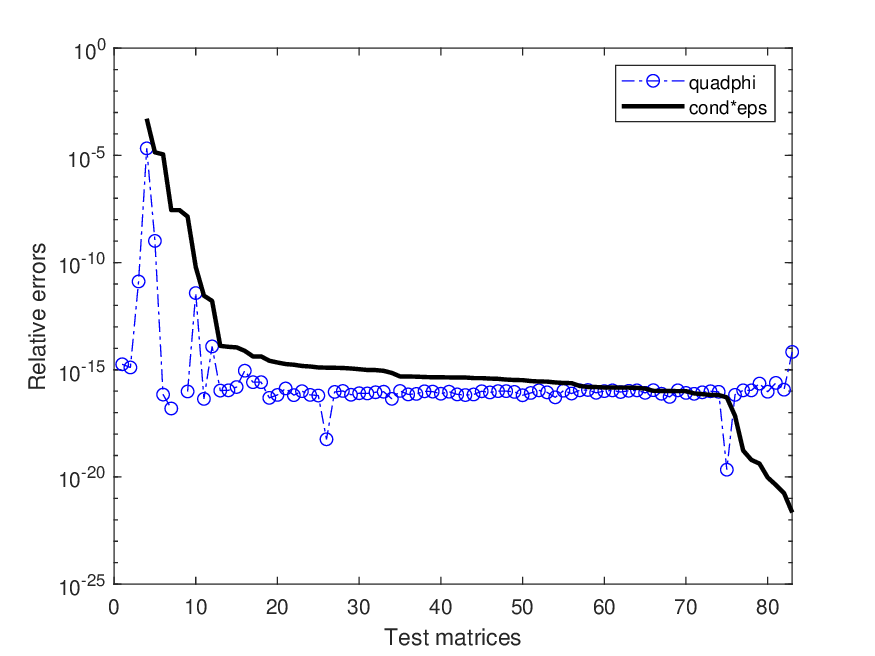}\\
\end{minipage}\\
\begin{minipage}{0.5\linewidth}
\centering
\includegraphics[width=6.5cm,height=5cm]{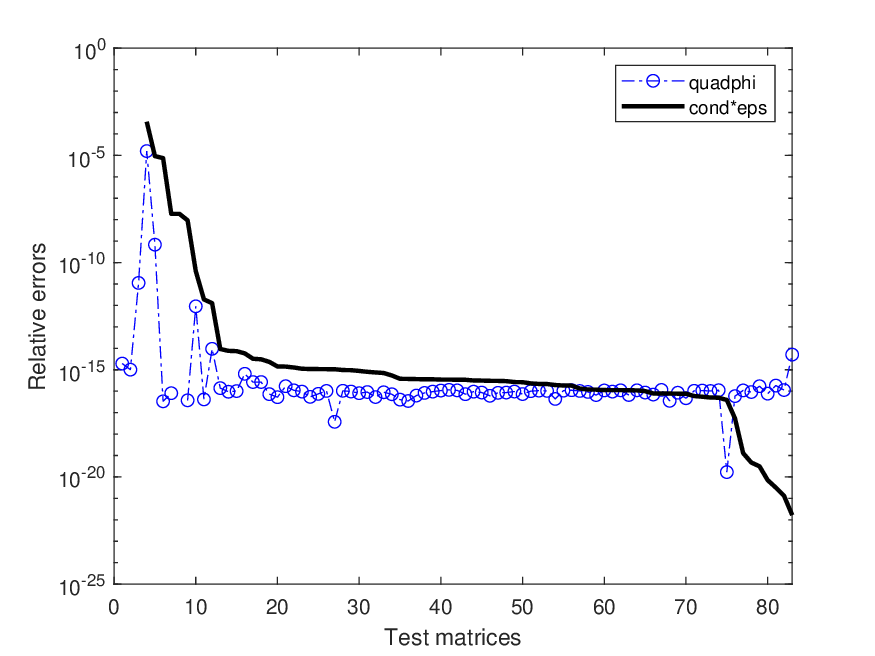}\\
\end{minipage}
\mbox{\hspace{-1.5cm}}
\begin{minipage}{0.5\linewidth}
\centering
\includegraphics[width=6.5cm,height=5cm]{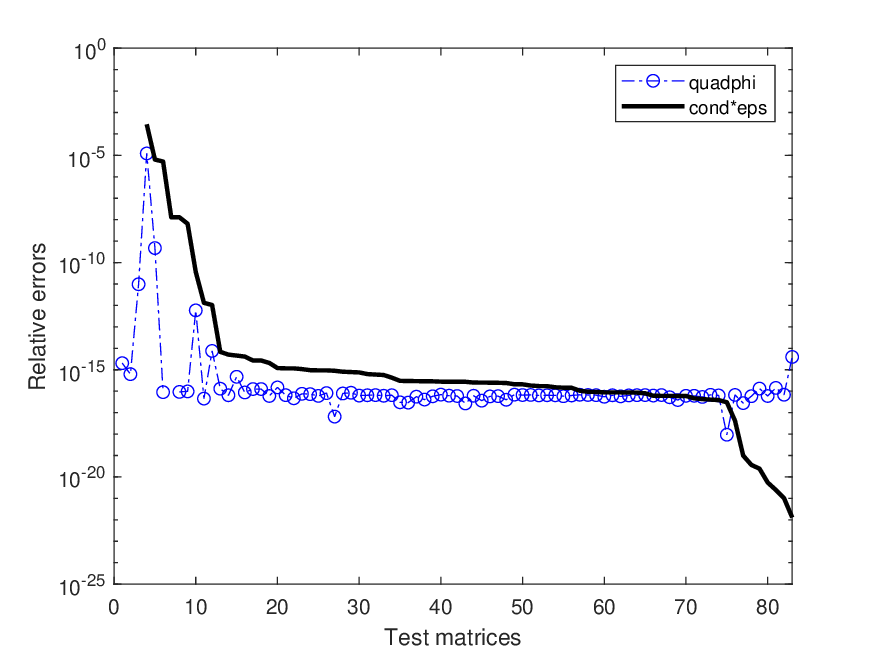}\\
\end{minipage}
\caption{Relative errors of \texttt{quadphi} for solving $\varphi_l(A)$ for $l=0, . . . , 7$ (Left to Right) of Experiment \ref{exa1}.}\label{fig4.1}
\end{figure}

\begin{figure}[H]
\begin{minipage}{0.5\linewidth}
\centering
\includegraphics[width=6.5cm,height=5cm]{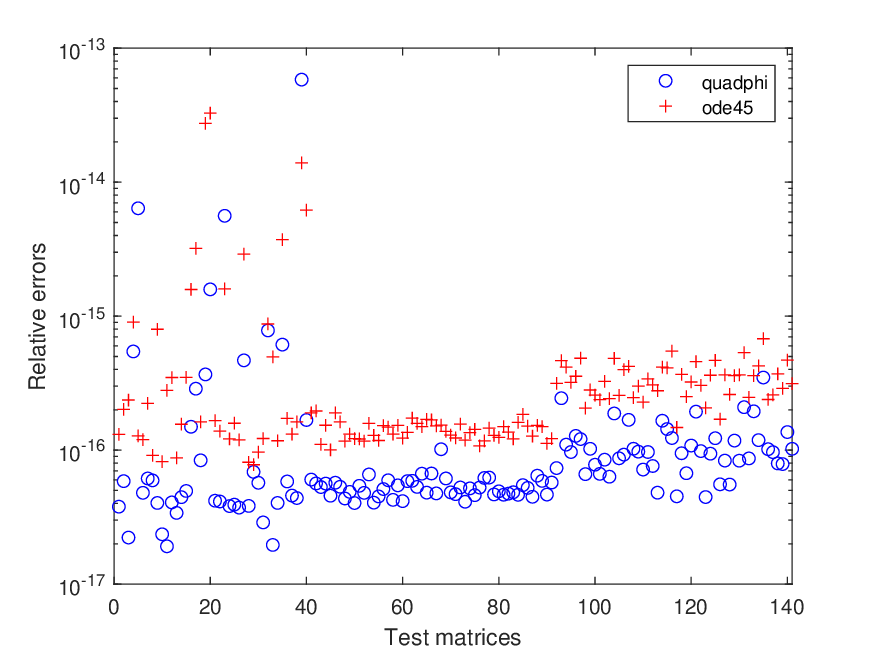}\\
\end{minipage}
\mbox{\hspace{-1.5cm}}
\begin{minipage}{0.5\linewidth}
\centering
\includegraphics[width=6.5cm,height=5cm]{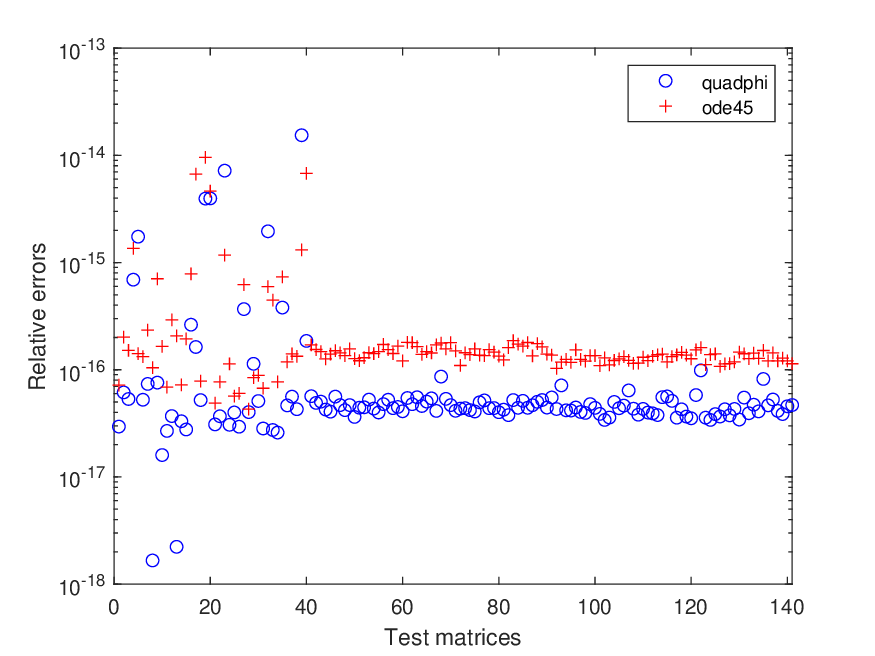}\\
\end{minipage}\\
\begin{minipage}{0.5\linewidth}
\centering
\includegraphics[width=6.5cm,height=5cm]{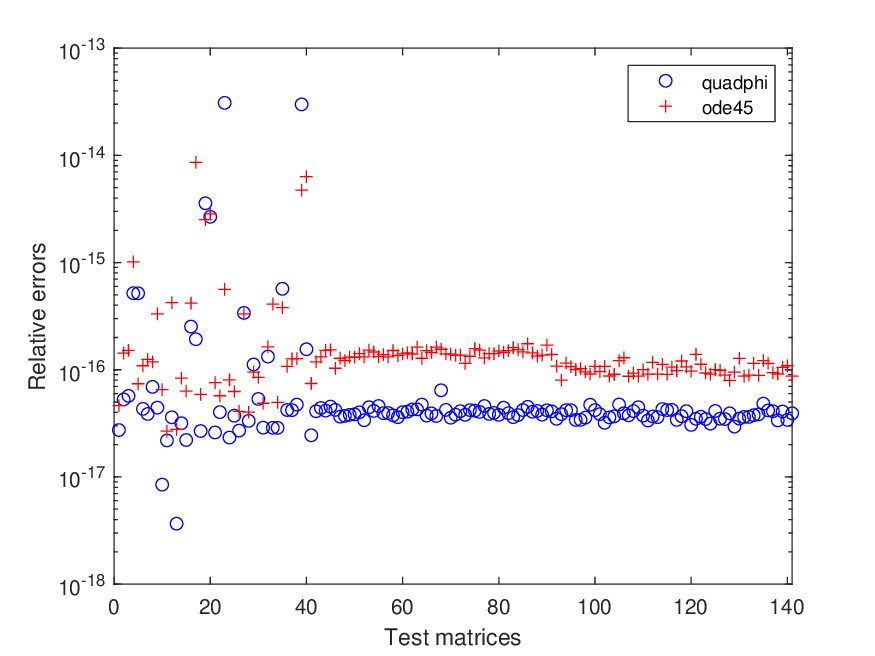}\\
\end{minipage}
\mbox{\hspace{-1.5cm}}
\begin{minipage}{0.5\linewidth}
\centering
\includegraphics[width=6.5cm,height=5cm]{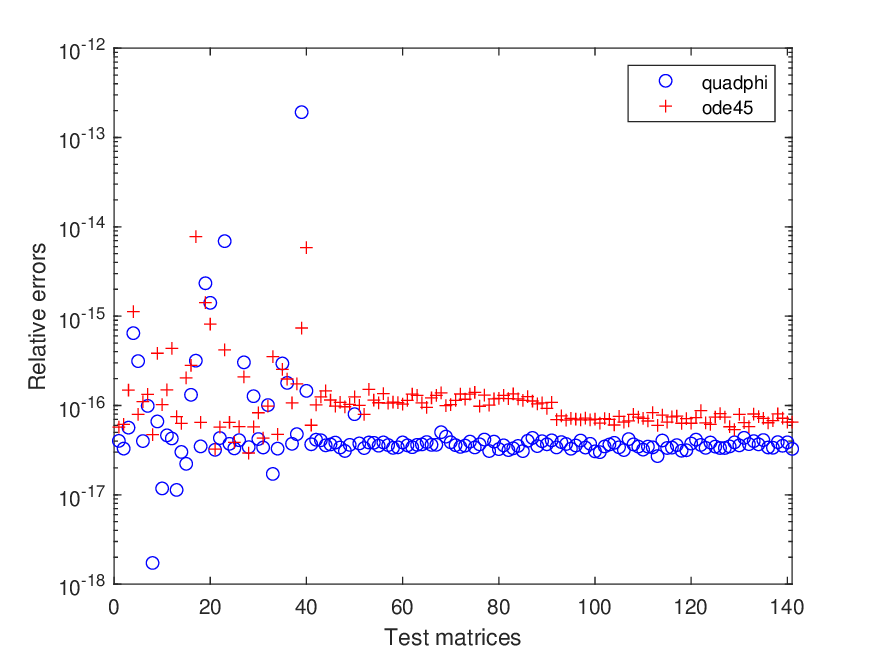}\\
\end{minipage}\\
\begin{minipage}{0.5\linewidth}
\centering
\includegraphics[width=6.5cm,height=5cm]{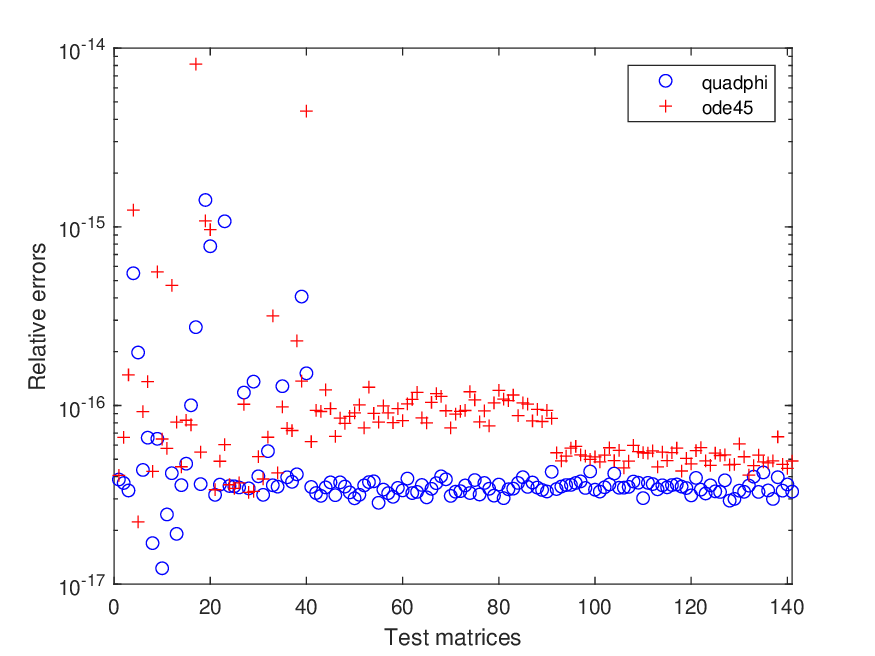}\\
\end{minipage}
\mbox{\hspace{-1.5cm}}
\begin{minipage}{0.5\linewidth}
\centering
\includegraphics[width=6.5cm,height=5cm]{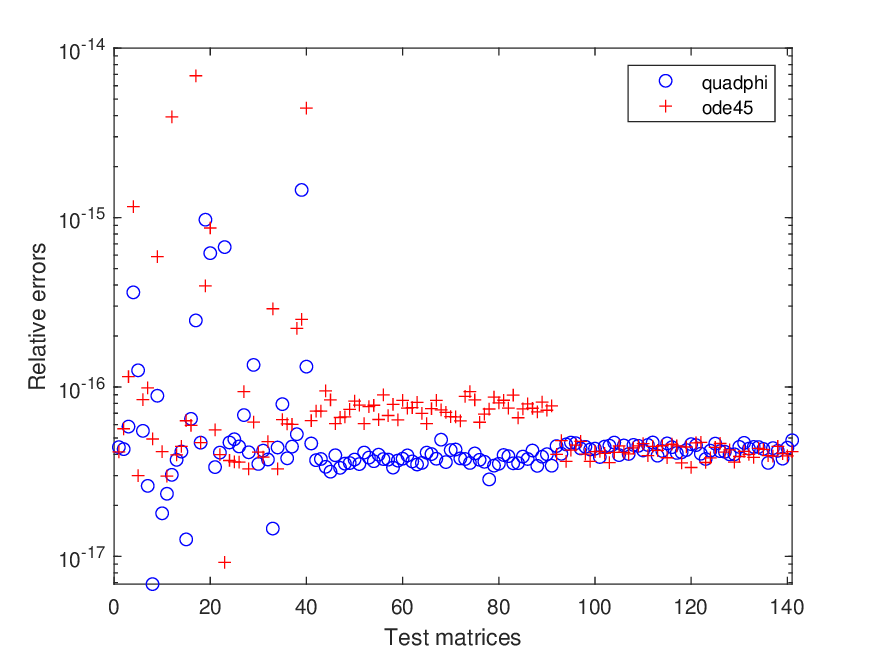}\\
\end{minipage}\\
\begin{minipage}{0.5\linewidth}
\centering
\includegraphics[width=6.5cm,height=5cm]{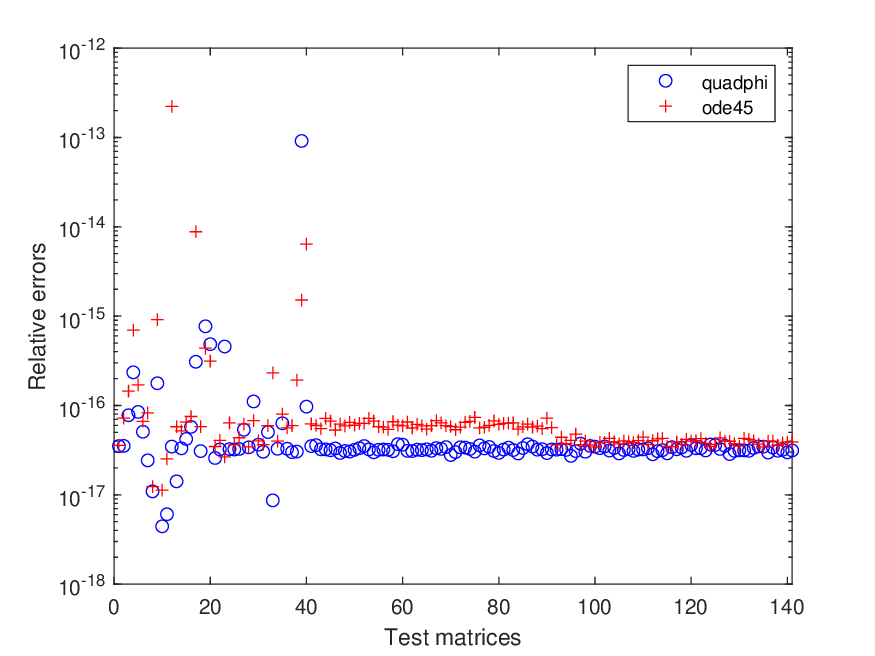}\\
\end{minipage}
\mbox{\hspace{-1.5cm}}
\begin{minipage}{0.5\linewidth}
\centering
\includegraphics[width=6.5cm,height=5cm]{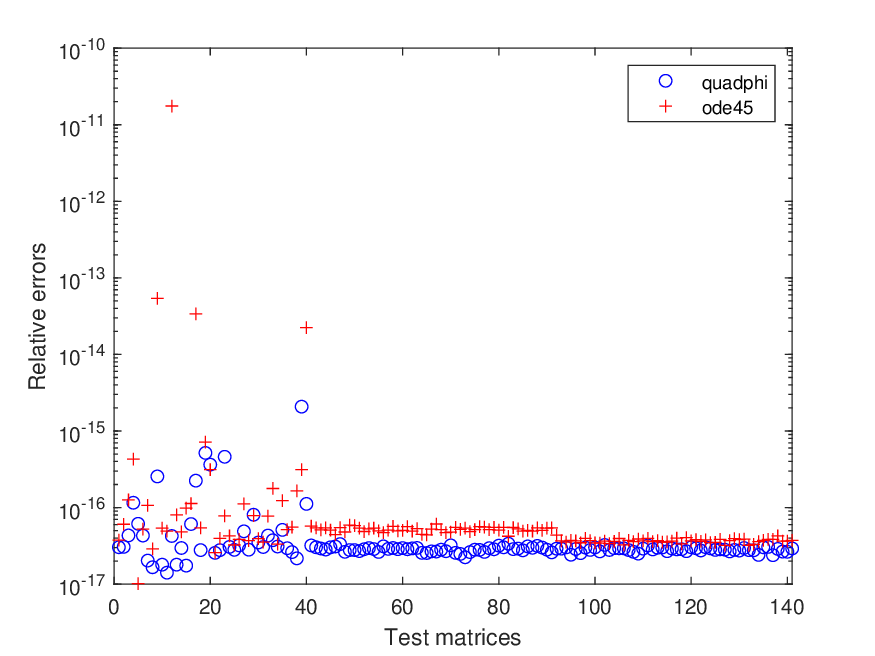}\\
\end{minipage}
\caption{ Relative errors of \texttt{quadphi} and \texttt{ode45} for solving $\phi_l(A)b$ for $l=0, . . . , 7$ (Left to Right) of Experiment \ref{exa2}.}\label{fig4.2}
\end{figure}

\begin{table}[htbp]
\setlength{\abovecaptionskip}{0.cm}
\setlength{\belowcaptionskip}{-0.3cm}
\caption{Execution times of \texttt{quadphi} and \texttt{ode45} for solving $\phi_l(A)b$ for $l=0, . . . , 7$  of Experiment \ref{exa2}.}\label{tab2.1}
\begin{center}
\begin{tabular*}{\textwidth}{@{\extracolsep{\fill}}@{~}c|ccccccccccc}
\toprule%
$l$ & $0$& $1$& $2$& $3$& $4$& $5$& $6$ & $7$ \\
  \hline
  \texttt{quadphi} &0.37   &0.46   &0.31   &0.45  &0.30   &0.32 &0.34   &0.36\\
 \hline
 \texttt{ode45}&84.42  &60.94  &65.60  &77.62  &68.90  &77.58 &71.77  &64.17\\
\bottomrule
\end{tabular*}
\end{center}
\end{table}

\bibliographystyle{elsarticle-num}
\bibliography{reference}

\end{document}